\documentclass[11pt]{article}

\usepackage{amsfonts,latexsym,amsmath,amssymb,amsthm}
\usepackage{fullpage}
\usepackage{hyperref}

\newtheorem{theorem}{Theorem}[section]
\newtheorem{proposition}[theorem]{Proposition}
\newtheorem{lemma}[theorem]{Lemma}

\numberwithin{equation}{section}

\newcommand{\R}{\mathbb R}
\newcommand{\Pp}{\mathbb P}
\newcommand{\conv}{\operatorname{conv}}
\newcommand{\vol}{\operatorname{vol}}
\newcommand{\supp}{\operatorname{supp}}
\newcommand{\rank}{\operatorname{rank}}
\newcommand{\BM}{d_{\mathrm{BM}}}
\newcommand{\eps}{\varepsilon}
\newcommand{\ls}{\leqslant}
\newcommand{\gr}{\geqslant}

\usepackage{setspace}
\begin{document}
\small

\title{\bf Gaussian polytopes with large Banach--Mazur distance to the cross-polytope}
\author{Antonios Hmadi}
\date{}
\maketitle

\begin{abstract}
\footnotesize
Let $B_1^n$ be the standard cross-polytope in $\R^n$, let $g_1,\ldots,g_m$ be independent standard Gaussian vectors in $\R^n$, and set $G_m=\conv\{\pm g_1,\ldots,\pm g_m\}$.
For $m=n^3$ it is proved that
$$ \Pp\left\{\BM(G_m,B_1^n)\gr c n^{5/8}(\ln n)^{-5/8}\right\}\gr 1-\frac2n $$
for a suitable absolute constant $c>0$.
This independently improves the polynomial exponent $4/7$ in Friedland's preceding work.
Independent concurrent work of Friedland which appeared after completion of the present manuscript obtains the same polynomial exponent with the stronger logarithmic factor $(\ln n)^{-1/4}$ by a different argument.
The proof uses Friedland's discretization and conditioning argument together with the $K/U$ decomposition.
A selected family of $K$ vectors is suppressed and the remaining $K$ vectors are quotiented out.
In the resulting quotient simultaneous bounds are proved for every top dimensional exterior product formed from the suppressed $K$ vectors and the $U$ vectors.
A Dvoretzky--Rogers selection after L\"owner normalization converts these determinant estimates into a bound for the minimum volume ellipsoid of the whole projected polytope and Maurey's empirical method then gives the required Gaussian measure estimate.
\end{abstract}

\section{Introduction}

The diameter of the Banach--Mazur compactum of $n$ dimensional normed spaces is of order $n$ by Gluskin's theorem~\cite{Gluskin81}.
A related more rigid problem posed by Pe\l czy\'nski, is to fix one of the two spaces and ask how far an arbitrary $n$ dimensional space can be from it.
For the cube or equivalently by polarity for the cross-polytope this asks for the asymptotic order of the radius of the Banach--Mazur compactum with the cube as fixed center; see~\cite{Pelczynski84,Szarek91}.

For $d\in\mathbb N$ and $1\ls p\ls\infty$, let $B_p^d$ denote the unit ball of $\ell_p^d$; thus, if $e_1,\ldots,e_n$ is the standard basis of $\R^n$, then $B_1^n=\conv\{\pm e_1,\ldots,\pm e_n\}$ is the cross-polytope and $B_\infty^n=[-1,1]^n$ is the cube.
Write $GL_n$ for the group of invertible linear operators on $\R^n$.
For centrally symmetric convex bodies $K,L\subset\R^n$ their Banach--Mazur distance is
$$ \BM(K,L)=\inf\{a\gr1:T(K)\subseteq L\subseteq aT(K)\ \hbox{for some }T\in GL_n\}. $$
The same notation is used for normed spaces and their unit balls.
Put $R_\infty(n):=\sup_K\BM(K,B_\infty^n)$, where the supremum is over all centrally symmetric convex bodies in $\R^n$.
By polarity, $R_\infty(n)=\sup_K\BM(K,B_1^n)$.
For general background on Banach--Mazur distance, polarity and the methods of asymptotic geometric analysis used below, see~\cite{AAGM15,AAGM21}.

The best general upper exponent is $5/6$, obtained by Giannopoulos~\cite{Giannopoulos95}.
Youssef later recovered this exponent, with an improved constant, through restricted invertibility and proportional Dvoretzky--Rogers factorization~\cite[Theorems~3.6 and~3.7]{Youssef14}.
The connection between this problem and Dvoretzky--Rogers factorization goes back to Bourgain and Szarek~\cite{BourgainSzarek88}.
In the proportional form, Szarek and Talagrand obtained a dependence of order $\varepsilon^{-2}$~\cite{SzarekTalagrand89}; Giannopoulos improved it first to $\varepsilon^{-3/2}$ and then to $\varepsilon^{-1}$~\cite{Giannopoulos95,Giannopoulos96}.
Youssef later recovered the $\varepsilon^{-1}$ dependence by normalized restricted invertibility~\cite[Section~3.1]{Youssef14}.
If $C_{\mathrm{DR}}(\varepsilon)$ denotes the best constant in the proportional Dvoretzky--Rogers factorization, the presently known estimates give
$$ c\varepsilon^{-1/2}\ls C_{\mathrm{DR}}(\varepsilon)\ls C\varepsilon^{-1}. $$
The lower obstruction follows from Rudelson's construction of convex bodies with many contact points~\cite{Rudelson97}.
The later restricted invertibility theorem of Naor and Youssef exhibits a square root dependence on the deficit for a fixed linear system but retains additional spectral information on that system and does not yield the corresponding $\varepsilon^{-1/2}$ estimate for arbitrary John decompositions~\cite{NaorYoussef17}.
Thus the optimal dependence in the proportional Dvoretzky--Rogers theorem remains open.
If the lower endpoint $\varepsilon^{-1/2}$ were attainable, the standard argument relating proportional Dvoretzky--Rogers factorization to the distance from the cube would give the upper exponent $3/4$.

On the lower side, Szarek~\cite{Szarek90} obtained $c\sqrt n\ln n$, Tikhomirov~\cite{Tikhomirov19} proved a lower bound of order $n^{5/9}$ up to logarithmic factors and Friedland~\cite{Friedland26} improved the exponent to $4/7$ for the Gaussian Gluskin model.
At the time the present manuscript was completed, this was the best publicly available polynomial exponent on the lower side, leaving a gap between $4/7$ and $5/6$.

Following Friedland~\cite{Friedland26}, consider the Gaussian Gluskin polytope.
Let $m=n^3$, let $\Gamma$ be an $n\times m$ matrix with independent standard Gaussian entries and write $G_m=\Gamma(B_1^m)=\conv\{\pm g_1,\ldots,\pm g_m\}$ where $g_1,\ldots,g_m$ are the columns of $\Gamma$.

\begin{theorem}\label{thm:main}
There exist absolute constants $c>0$ and $n_0\in\mathbb N$ such that, for every $n\gr n_0$ and $m=n^3$,
$$ \Pp\left\{\BM(G_m,B_1^n)\gr c n^{5/8}(\ln n)^{-5/8}\right\}\gr1-\frac2n. $$
Consequently,
$$ R_\infty(n)\gr c n^{5/8}(\ln n)^{-5/8}. $$
\end{theorem}

\medskip
\noindent \textbf{Chronology and independent concurrent work.}
The present manuscript including Theorem~\ref{thm:main} and its proof was completed and submitted for publication before the author learned of Friedland's concurrent $5/8$ result.
On August 17, 2026, before the present paper had appeared publicly the author communicated the completed manuscript to Friedland.
In response Friedland informed the author that he had independently obtained the same polynomial exponent and shared a circulation copy of his manuscript.
The version subsequently posted as~\cite{FriedlandFiveEighths} proves the estimate
$$ d_{\mathrm{BM}}(G_m,B_1^n)\gr c n^{5/8}(\ln n)^{-1/4}. $$
Friedland records in~\cite{FriedlandFiveEighths} that his two event proof architecture and an earlier bound of order $n^{5/8}(\ln n)^{-3/8}$ had been obtained before he received the present manuscript and that a later optimization of the small coordinate estimate improved the logarithmic factor to $(\ln n)^{-1/4}$.
The two proofs share the discretization, conditioning and $K/U$ framework of Friedland's preceding paper~\cite{Friedland26} but diverge after this common reduction.
The present proof proceeds by a different mechanism described below.
Friedland's proof instead introduces two uniform quotient events, one controls successive directions of the big coordinate parts after every admissible preceding span while the other places the whole small coordinate cloud near a low dimensional subspace after every admissible quotient; suppression, a local Maurey argument and Gram--Schmidt width estimates then give the Gaussian measure estimate.
The two $5/8$ results were developed independently.

To explain the argument, fix a candidate distance $\rho$ and put $\Lambda=\ln(n\rho)$.
Friedland's discretization and conditional estimate are used; the latter is based on Tikhomirov's Lemma~3.1 argument~\cite[Sections~2.3--2.5]{Friedland26}\cite[Lemma~3.1]{Tikhomirov19}.
The columns of a discretized operator are split at level $1/s$ into their $K$ and $U$ parts.
Instead of separating the argument according to the number of $U$ vectors in a representation, choose $r$ of the $K$ vectors, suppress them by the factor $r/n$, and quotient out the remaining $n-r$ of the $K$ vectors.
The quotient has dimension $r$.
A simultaneous Gaussian estimate controls every $r$ fold exterior product formed from the surviving suppressed $K$ vectors and the $U$ vectors.
The entropy of the coefficient classes is paid by the $e^{-cr^2}$ determinant tail provided $ r^2\gr Csn\Lambda. $
The projected $K$ vectors have coefficient scale at most $r/n$, while the projected $U$ vectors have coefficient scale at most $s^{-1/2}$.
Writing $ \delta=\max\left\{ r/n,1/{\sqrt s}\right\}, $ the exterior product estimate and a Dvoretzky--Rogers selection give a bound of order $\sqrt r\,\delta$ for the volume radius of the minimum volume containing ellipsoid of the whole projected polytope.
Maurey's empirical method after this affine normalization gives the second condition $ \rho\delta\sqrt{\Lambda/r}\ls c. $
Balancing the two coefficient scales and the entropy condition gives
$$ s\approx n^{1/2}\Lambda^{-1/2},\qquad r\approx n^{3/4}\Lambda^{1/4},\qquad \rho\approx n^{5/8}\Lambda^{-5/8}. $$
Since $\Lambda\approx\ln n$, this yields Theorem~\ref{thm:main}.

\noindent \textbf{Organization of the paper.}
Section~2 recalls the part of Friedland's reduction used here and fixes the notation.
Section~3 proves the simultaneous exterior product estimate used in the quotient argument.
Section~4 combines suppression with the L\"owner and Maurey estimates to obtain the uniform Gaussian measure bound.
Section~5 contains the choice of parameters and the proof of Theorem~\ref{thm:main}.

\section{Reduction and notation}

The notation $|\cdot|$ and $\langle\cdot,\cdot\rangle$ is used for the Euclidean norm and inner product, $S^{d-1}$ for the Euclidean unit sphere in $\R^d$, $[N]:=\{1,\ldots,N\}$, and $e_1,\ldots,e_d$ for the standard basis of $\R^d$.
Lebesgue measure on a $d$ dimensional Euclidean space is denoted by $\vol_d$ and the indicator of an event or set $E$ by $\mathbf 1_E$.
If $H$ is a Euclidean subspace, $P_H$ denotes the orthogonal projection onto $H$ and $\gamma_H$ the standard Gaussian measure on $H$; when $H=\R^d$, the notation $\gamma_d$ is used.
For a vector $x$, $\supp(x):=\{j:x_j\ne0\}$, while $\operatorname{range}M$ and $\rank M$ denote the range and rank of a matrix $M$.
For a matrix $M$, $M^*$ denotes its transpose, $\|M\|_{\mathrm{HS}}$ its Hilbert--Schmidt norm, and $I_d$ the identity operator on $\R^d$.
For vectors $x,y$, $x\otimes y$ denotes the rank one operator $z\mapsto\langle z,y\rangle x$, and $\operatorname{ext}(K)$ denotes the set of extreme points of a convex body $K$.
If $x_1,\ldots,x_k$ lie in a Euclidean space, $\|x_1\wedge\cdots\wedge x_k\|$ denotes the $k$ dimensional volume of the parallelepiped generated by them.
For $S\subset[m]$, $\Gamma_S$ denotes the submatrix of $\Gamma$ formed by the columns indexed by $S$, while $A_S$ denotes the submatrix of an $m\times n$ matrix $A$ formed by the rows indexed by $S$.
Throughout, $c,C,c_1,C_1,\ldots$ are positive absolute constants whose values may change from line to line.
The notation $a\approx b$ means that there are absolute constants $c_1,c_2>0$ such that $c_1a\ls b\ls c_2a$.

Fix $\rho\gr1$ and set
\begin{equation}\label{eq:lambda} \Lambda:=\ln(n\rho),\qquad \eps:=\frac1{\rho n^2}.\end{equation}
Since $m=n^3$ and $\eps\rho n^2=1$, Friedland's discretization lemma~\cite[Lemma~2.2(ii)--(iii)]{Friedland26} provides a finite class $\mathcal A_\eps$ of $m\times n$ matrices whose columns are supported on at most $n$ coordinates, have coordinates in $\eps\mathbb Z$, and have $\ell_1$ norm at most one, such that
\begin{equation}\label{eq:Anet} \ln|\mathcal A_\eps|\ls Cn^2\Lambda,\end{equation}
and
$$ \{\BM(G_m,B_1^n)\ls\rho\}\subseteq \bigcup_{A\in\mathcal A_\eps}\{G_m\subset2\rho\Gamma A(B_1^n)\}. $$

Fix $A\in\mathcal A_\eps$.
Write $A^{(1)},\ldots,A^{(n)}$ for its columns and set
$$ S(A):=\bigcup_{i=1}^n\supp(A^{(i)}),\qquad N(A):=m-|S(A)|. $$
Then $|S(A)|\ls n^2$ and $N(A)\gr m-n^2$.
Let $\mathcal F_A$ be the sigma-field generated by the Gaussian columns $g_j$, $j\in S(A)$.
Following Friedland's terminology, after conditioning on $\mathcal F_A$ the exposed matrix is denoted by $\Gamma_A:=\Gamma_{S(A)} $ and $\mathcal K_A(\rho;\Gamma_A):=2\rho\,\Gamma_A A_{S(A)}(B_1^n)$.
Friedland's Lemma~2.4, called the Powering Lemma there~\cite{Friedland26}, implements the conditioning argument of Tikhomirov~\cite[Lemma~3.1]{Tikhomirov19} and gives, conditionally on $\mathcal F_A$,
\begin{equation}\label{eq:powering} \Pp\bigl(G_m\subset2\rho\Gamma A(B_1^n)\mid\mathcal F_A\bigr) \ls \gamma_n\bigl(\mathcal K_A(\rho;\Gamma_A)\bigr)^{N(A)}.\end{equation}

Let $s\in\{1,\ldots,n\}$.
Following the $K/U$ decomposition used by Tikhomirov~\cite[Section~3]{Tikhomirov19} and Friedland~\cite[Section~2.5]{Friedland26}, write $A^{(i)}=a_i^{(K)}+a_i^{(U)}$ by splitting the column $A^{(i)}$ at level $1/s$.
Thus $(a_i^{(K)})_j=A^{(i)}_j\mathbf 1_{\{|A^{(i)}_j|\gr1/s\}}$ and $a_i^{(U)}=A^{(i)}-a_i^{(K)}$.
Since every column of $A$ has $\ell_1$ norm at most one, at most $s$ coordinates can have absolute value at least $1/s$.
Hence
$$ |\supp(a_i^{(K)})|\ls s,\qquad \|a_i^{(U)}\|_1\ls1,\qquad \|a_i^{(U)}\|_\infty\ls\frac1s. $$
The symmetric polytope
\begin{equation}\label{eq:Cs} \mathcal C_s:=B_1^m\cap s^{-1}B_\infty^m\end{equation}
will be used repeatedly.
Note that $a_i^{(U)}\in\mathcal C_s$ and $|a_i^{(U)}|\ls s^{-1/2}$.
Put $v_i:=\Gamma a_i^{(K)}$ and $u_i:=\Gamma a_i^{(U)}$.
These vectors are $\mathcal F_A$-measurable because all the coefficient vectors are supported in $S(A)$; on the fiber determined by $\Gamma_A$, they are obtained by applying $\Gamma_A$ to the restrictions of $a_i^{(K)}$ and $a_i^{(U)}$ to $S(A)$.

\section{The exterior product estimate}

Two elementary counting facts are recorded first.
Fix an extreme point $x$ of $\mathcal C_s$ and a closed orthant containing $x$.
Then $x$ is also an extreme point of the intersection of $\mathcal C_s$ with this orthant.
After changing signs and multiplying by $s$, this intersection becomes the polytope
$$ \left\{y\in\R^m:0\ls y_j\ls1,\ \sum_{j=1}^m y_j\ls s\right\}. $$
Its vertices are the $0$-$1$ vectors with at most $s$ coordinates equal to one.
Indeed, if $\sum_jy_j<s$ and some coordinate satisfies $0<y_j<1$, that coordinate can be perturbed in both directions.
If $\sum_jy_j=s$ and two coordinates are fractional, they can be perturbed in opposite directions while preserving the sum.
If $\sum_jy_j=s$ and exactly one coordinate is fractional, the integrality of $s$ and of all the remaining coordinates gives a contradiction.
It follows that every extreme point of $\mathcal C_s$ has the form
$$ \frac1s\sum_{j\in J}\varepsilon_j e_j, \qquad |J|\ls s,\quad \varepsilon_j\in\{-1,1\}, $$
and hence
\begin{equation}\label{eq:extcount} |\operatorname{ext}(\mathcal C_s)| \ls\sum_{\ell=0}^s2^\ell\binom m\ell \ls2^s\sum_{\ell=0}^s\binom m\ell\ls\left(\frac{2em}{s}\right)^s.\end{equation}

For $1\ls p\ls n$, let $\mathcal T_p$ be the set of $m\times p$ matrices whose columns have coordinates in $\eps\mathbb Z$, are supported on at most $s$ coordinates and have $\ell_1$ norm at most one.
Set also $\mathcal T_0:=\{0\}$, where $0$ denotes the unique linear map from the zero dimensional space into $\R^m$.
This convention is used only for the endpoint $r=n$, where $n-r=0$.
In Appendix~B.1 Friedland considers the smaller class of matrices which occur as $K$ submatrices and obtains the same columnwise cardinality bound in the range needed there~\cite[Appendix~B.1]{Friedland26}.
The larger deterministic class above is used.
For $p\gr1$, one column has at most
$$ \sum_{\ell=0}^s\binom m\ell\ls\left(\frac{em}{s}\right)^s $$
choices for the support.
Once the support is fixed every nonzero coordinate lies on the $\eps$ grid in $[-1,1]$, so there are at most $(C/\eps)^s$ possible values.
Hence
\begin{equation}\label{eq:Tcount} |\mathcal T_p|\ls\left(\frac{Cm}{s\eps}\right)^{sp}.\end{equation}
By $m=n^3$, $\eps^{-1}=\rho n^2$ and~\eqref{eq:lambda}, it follows that $ \ln|\mathcal T_p|\ls Csp\Lambda$ and $\ln|\operatorname{ext}(\mathcal C_s)|\ls Cs\Lambda. $

Let
$$ \mathcal D_s:=\{x\in(\eps\mathbb Z)^m:|\supp(x)|\ls s,\ \|x\|_1\ls1\}. $$
Every $K$ part $a_i^{(K)}$ belongs to $\mathcal D_s$.
The same one column count gives
\begin{equation}\label{eq:Dcount} \ln|\mathcal D_s|\ls Cs\Lambda.\end{equation}

The next lemma is an elementary Gaussian estimate; the proof is included because the $e^{-cr^2}$ tail is the point that pays for the simultaneous entropy below.

\begin{lemma}\label{lem:dettail}
Let $q,r\in\mathbb N$.
Let $G:\R^q\to\R^r$ be a random operator whose matrix has independent standard Gaussian entries and let $d_1,\ldots,d_r\in\R^q$ satisfy $|d_i|\ls a_i$, where $a_i>0$.
There are absolute constants $D_0,c>0$ such that for every $D\gr D_0$,
$$ \Pp\left\{\|Gd_1\wedge\cdots\wedge Gd_r\|>(D\sqrt r)^r\prod_{i=1}^ra_i\right\}\ls e^{-cD^2r^2}. $$
\end{lemma}

\begin{proof}
Let $D_1=[d_1\ \cdots\ d_r]$.
The rows of the $r\times r$ matrix $GD_1$ are independent centered Gaussian vectors in $\R^r$ with covariance $D_1^*D_1$.
Hence $GD_1$ has the same distribution as $G_r(D_1^*D_1)^{1/2}$, where $G_r$ is a standard $r\times r$ Gaussian matrix.
Therefore $\|Gd_1\wedge\cdots\wedge Gd_r\|$ has the same distribution as $|\det G_r|\sqrt{\det(D_1^*D_1)}$.
Hadamard's inequality gives $\sqrt{\det(D_1^*D_1)}\ls\prod_{i=1}^r|d_i|\ls\prod_{i=1}^ra_i$.
By the inequality between the arithmetic and geometric means applied to the squared singular values, $ |\det G_r|^{2/r}\ls  {\|G_r\|_{\mathrm{HS}}^2}/{r}. $
It follows that
$$ \Pp\left\{\|Gd_1\wedge\cdots\wedge Gd_r\|>(D\sqrt r)^r\prod_{i=1}^ra_i\right\}\ls\Pp\left\{\|G_r\|_{\mathrm{HS}}^2>D^2r^2\right\}. $$
Since $\|G_r\|_{\mathrm{HS}}^2$ has the $\chi^2_{r^2}$ distribution, the Chernoff bound gives
$$ \Pp\left\{\|G_r\|_{\mathrm{HS}}^2>D^2r^2\right\}\ls\exp\left(-\frac{r^2}{2}(D^2-1-\ln D^2)\right)\ls e^{-cD^2r^2} $$
for $D\gr D_0$, after choosing $D_0$ sufficiently large.
\end{proof}

The following proposition is the probabilistic input used in the proof.
The quotient is realized by an orthogonal decomposition of the coefficient space in order to retain an independent Gaussian block.
The formulation includes rank deficient coefficient matrices by completing their ranges deterministically before the Gaussian matrix is sampled.

\begin{proposition}\label{prop:wedge}
There are absolute constants $C_0,C_1,c>0$ with the following property.
Let $1\ls r\ls n$ and suppose
\begin{equation}\label{eq:wedgecondition} r^2\gr C_0sn\Lambda.\end{equation}
For every $B\in\mathcal T_{n-r}$ choose deterministically an $(n-r)$ dimensional subspace $\widehat E_B\subset\R^m$ such that $\operatorname{range}B\subset\widehat E_B$ and put $H_B=(\Gamma\widehat E_B)^\perp$.
Then with probability at least $1-e^{-cr^2}$ simultaneously for every $B\in\mathcal T_{n-r}$, every integer $0\ls p\ls r$, every $x_1,\ldots,x_p\in\mathcal D_s$ and every $z_1,\ldots,z_{r-p}\in\mathcal C_s$,
\begin{equation}\label{eq:wedge} \left\|\bigwedge_{i=1}^pP_{H_B}\Gamma x_i\wedge\bigwedge_{j=1}^{r-p}P_{H_B}\Gamma z_j\right\|\ls(C_1\sqrt r)^rs^{-(r-p)/2}.\end{equation}
Here an empty exterior product is omitted.
On the same event, $\dim H_B=r$ for every $B\in\mathcal T_{n-r}$.
\end{proposition}

\begin{proof}
For every fixed $B\in\mathcal T_{n-r}$, the restriction of $\Gamma$ to the fixed subspace $\widehat E_B$ is injective almost surely.
Since $\mathcal T_{n-r}$ is finite this holds simultaneously for every $B$ outside a null event.
Work on this event.

Fix $B\in\mathcal T_{n-r}$.
Choose matrices $U_B,V_B$ whose columns are orthonormal bases of $\widehat E_B$ and $\widehat E_B^\perp$ respectively and put $X_B=\Gamma U_B$ and $Y_B=\Gamma V_B$.
Since $[U_B\ V_B]$ is a fixed orthogonal matrix on $\R^m$, rotational invariance shows that $\Gamma[U_B\ V_B]=[X_B\ Y_B]$ is again a standard Gaussian matrix.
In particular, the two column blocks $X_B$ and $Y_B$ are independent standard Gaussian matrices.
The matrix $X_B$ has rank $n-r$ and hence $\dim H_B=r$.

Condition on a realization of $X_B$ and choose an orthonormal basis matrix $R_B$ of $H_B$.
Then $R_B$ is fixed while $Y_B$ remains an independent standard Gaussian matrix.
Since $R_B^*R_B=I_r$ the operator $G_B:=R_B^*Y_B:\R^{m-n+r}\to\R^r$ is standard Gaussian.
For $w\in\R^m$, write $w=U_Bu+V_Bv$ where $u=U_B^*w$ and $v=V_B^*w$.
Since $R_B^*X_B=0$ and $R_B^*:H_B\to\R^r$ is an isometry, $ R_B^*P_{H_B}\Gamma w=R_B^*Y_BV_B^*w=G_BV_B^*w. $

Fix $0\ls p\ls r$, $x_1,\ldots,x_p\in\mathcal D_s$ and $z_1,\ldots,z_{r-p}\in\mathcal C_s$.
For every $i\ls p$, $ |V_B^*x_i|\ls|x_i|\ls\|x_i\|_1\ls1, $ while for every $j\ls r-p$, $ |V_B^*z_j|\ls|z_j|\ls\sqrt{\|z_j\|_1\|z_j\|_\infty}\ls s^{-1/2}. $
Lemma~\ref{lem:dettail} therefore gives, conditionally on $X_B$,
$$\Pp\left\{\left. \left\|\bigwedge_{i=1}^pP_{H_B}\Gamma x_i\wedge\bigwedge_{j=1}^{r-p}P_{H_B}\Gamma z_j\right\| >(D\sqrt r)^rs^{-(r-p)/2} \,\right|\,X_B\right\} \ls e^{-cD^2r^2}.$$
Averaging over $X_B$ gives the same unconditional estimate for every fixed choice of $B,p,x_1,\ldots,x_p$ and $z_1,\ldots,z_{r-p}$.

For fixed $B$, $p$, $x_1,\ldots,x_p$ and $\Gamma$ the exterior product norm is convex in each $z_j$ separately.
Successively maximizing in $z_1,\ldots,z_{r-p}$ therefore gives
$$\sup_{z_1,\ldots,z_{r-p}\in\mathcal C_s} \left\|\bigwedge_{i=1}^pP_{H_B}\Gamma x_i\wedge\bigwedge_{j=1}^{r-p}P_{H_B}\Gamma z_j\right\| = \max_{z_1,\ldots,z_{r-p}\in\operatorname{ext}(\mathcal C_s)} \left\|\bigwedge_{i=1}^pP_{H_B}\Gamma x_i\wedge\bigwedge_{j=1}^{r-p}P_{H_B}\Gamma z_j\right\|.$$
Thus it is enough to take a union bound over $B\in\mathcal T_{n-r}$, $x_1,\ldots,x_p\in\mathcal D_s$ and $z_1,\ldots,z_{r-p}\in\operatorname{ext}(\mathcal C_s)$.
By~\eqref{eq:Tcount}, \eqref{eq:Dcount} and~\eqref{eq:extcount}, for each fixed $p$ the logarithm of the number of choices is at most $ Cs(n-r)\Lambda+Csp\Lambda+Cs(r-p)\Lambda\ls Csn\Lambda. $
After summing over $0\ls p\ls r$, the probability of failure is at most $ (r+1)\exp(Csn\Lambda-cD^2r^2). $
By~\eqref{eq:wedgecondition}, $ Csn\Lambda\ls({C}/{C_0})r^2, $ while $\ln(r+1)\ls r^2$.
Choose first $D\gr D_0$ sufficiently large and set $C_1=D$ and then choose $C_0$ sufficiently large.
The last probability is then at most $e^{-c'r^2}$.
Changing the absolute constant $c$ completes the proof.
\end{proof}

\section{Suppression and the Gaussian measure estimate}

For a full dimensional symmetric convex body $Q\subset\R^k$, let $\mathcal E_L(Q)=TB_2^k$ be its minimum volume containing ellipsoid, where $T$ is positive definite, and put $R_L(Q):=(\det T)^{1/k}$.
For a symmetric convex body $K$, write $K^\circ=\{y\in\R^k:|\langle x,y\rangle|\ls1\ \hbox{for every }x\in K\}$ for its polar body and $h_K(u)=\sup_{x\in K}\langle x,u\rangle$ for its support function.

The next deterministic lemma converts the exterior product bound into the affine scale used later.
The successive selection from the John contact points in its proof is the classical Dvoretzky--Rogers argument~\cite{DvoretzkyRogers50}; it is included in full in the present L\"owner formulation.
The John decomposition is used in the symmetric form recorded in~\cite[Theorem~2.1.15]{AAGM15}.

\begin{lemma}\label{lem:loewner}
Let $Q=\conv\{\pm x_1,\ldots,\pm x_N\}\subset\R^r$ be full dimensional and suppose, for some $\Delta>0$, that $|\det(x_{i_1},\ldots,x_{i_r})|\ls\Delta^r$ for every $i_1,\ldots,i_r$.
Then
\begin{equation}\label{eq:loewner} R_L(Q)\ls\sqrt e\,\Delta.\end{equation}
\end{lemma}

\begin{proof}
Let $\mathcal E_L(Q)=TB_2^r$ and put $z_i=T^{-1}x_i$ and $Q'=T^{-1}Q$.
Then $Q'\subset B_2^r$ and $B_2^r$ is the minimum volume containing ellipsoid of $Q'$.
By polarity, $B_2^r$ is the maximal volume ellipsoid contained in $(Q')^\circ$.
The symmetric form of John's theorem therefore gives contact points $u_j\in\partial((Q')^\circ)\cap S^{r-1}$ and positive numbers $a_j$ such that
\begin{equation}\label{eq:john} \sum_j a_j u_j\otimes u_j=I_r, \qquad \sum_j a_j=r.\end{equation}
The contact points belong to $Q'\cap S^{r-1}$.
Since $Q'=\conv\{\pm z_1,\ldots,\pm z_N\}\subset B_2^r$, the strict convexity of $B_2^r$ shows that every such contact point is one of the signed generators $\pm z_i$.

The classical Dvoretzky--Rogers selection is now performed.
Select $r$ contact points successively.
Suppose $u_{j_1},\ldots,u_{j_{\ell-1}}$ have been chosen and let $F$ be the orthogonal complement of their span.
Taking the trace of $P_F$ in~\eqref{eq:john} gives
$$ r-\ell+1=\sum_j a_j|P_Fu_j|^2. $$
Since $\sum_j a_j=r$, some contact point satisfies $|P_Fu_j|^2\gr(r-\ell+1)/r$.
This projection is nonzero, so the chosen point is not in the span of the preceding ones.
The Gram--Schmidt formula therefore gives distinct signed generators $z_{i_1},\ldots,z_{i_r}$ such that
$$ |\det(z_{i_1},\ldots,z_{i_r})| \gr\sqrt{\frac{r!}{r^r}}\gr e^{-r/2}, $$
where the last inequality follows from $r!\gr(r/e)^r$.
Consequently,
$$ \Delta^r\gr|\det(x_{i_1},\ldots,x_{i_r})| =\det(T)|\det(z_{i_1},\ldots,z_{i_r})| \gr e^{-r/2}\det T, $$
which proves~\eqref{eq:loewner}.
\end{proof}

The sampling step in the next lemma is Maurey's empirical method; the same Hilbert space approximation as in Friedland~\cite[Lemma~4.2]{Friedland26} is used, where reference is made to Pisier's exposition of Maurey's argument~\cite{Pisier81}.
The $\ell_1$ support reduction and the Gaussian density and cross-polytope volume estimates are the elementary tools recorded in~\cite[Lemma~A.7 and Corollary~A.8]{Friedland26} and~\cite[Lemmas~A.9--A.10]{Friedland26}.
The short arguments are reproduced because the affine normalization by the L\"owner ellipsoid is the point used here.

\begin{lemma}\label{lem:affine}
Let $Q=\conv\{\pm x_1,\ldots,\pm x_N\}\subset\R^r$ be full dimensional.
For every $1\ls t\ls r$ and every $\rho>0$,
\begin{equation}\label{eq:affine} \gamma_r(\rho Q) \ls\binom Nt\binom{t+r}{r} \left(\frac{C\rho R_L(Q)}{\sqrt{rt}}\right)^r.\end{equation}
\end{lemma}

\begin{proof}
Let $\mathcal E_L(Q)=TB_2^r$, put $z_i=T^{-1}x_i$ and set $Q'=T^{-1}Q$.
Then $Q'\subset B_2^r$.
First, the approximation used below is proved.
Since $Q'=\conv\{\pm z_1,\ldots,\pm z_N\}$, for every $y\in Q'$ one may write $y=\sum_i\lambda_i\varepsilon_i z_i$, where $\lambda_i\gr0$, $\sum_i\lambda_i\ls1$ and $\varepsilon_i\in\{-1,1\}$.
Let $Z$ equal $\varepsilon_i z_i$ with probability $\lambda_i$ and $0$ with the remaining probability.
Then $\mathbb EZ=y$ and, since $z_i\in B_2^r$, $\mathbb E|Z|^2\ls\sum_i\lambda_i\ls1$.
If $Z_1,\ldots,Z_t$ are independent copies of $Z$, then
$$ \mathbb E\left|y-\frac1t\sum_{j=1}^tZ_j\right|^2=\operatorname{Var}\left(\frac1t\sum_{j=1}^tZ_j\right)=\frac1t\operatorname{Var}(Z)=\frac1t\bigl(\mathbb E|Z|^2-|y|^2\bigr)\ls\frac1t. $$
Hence some realization has distance at most $t^{-1/2}$ from $y$.
Its average belongs to $Q'_S:=\conv\{\pm z_i:i\in S\}$ for a set $S\subset[N]$ with $|S|\ls t$.
Since $Q$ is full dimensional, $N\gr r\gr t$, so padding $S$ if necessary gives
\begin{equation}\label{eq:Maureycover} Q'\subset\bigcup_{|S|=t}\left(Q'_S+t^{-1/2}B_2^r\right).\end{equation}

Fix $S$ and let $Z_S:\R^t\to\R^r$ have columns $(z_i)_{i\in S}$.
If $x\in Q'_S$, then $x=Z_Sa$ for some $\|a\|_1\ls1$.
If $y\in t^{-1/2}B_2^r$, then $B_2^r\subset\sqrt r B_1^r$ gives $y=\sqrt{r/t}\,b$ for some $\|b\|_1\ls1$.
Hence $x+y=M(a,b)$, where $M=[Z_S\ \ \sqrt{r/t}\,I_r]$, and $\|(a,b)\|_1\ls2$.
Therefore $Q'_S+t^{-1/2}B_2^r\subset2M(B_1^{t+r})$.

Every point of $M(B_1^{t+r})$ has a representation supported on at most $r$ signed columns of $M$.
Indeed, write $x=\sum_{j=1}^{t+r}\lambda_j w_j$, where $\lambda_j\gr0$, $\sum_j\lambda_j\ls1$, and each $w_j$ is a signed column of $M$.
If more than $r$ coefficients are nonzero, the corresponding vectors are linearly dependent, so $\sum_j\theta_jw_j=0$ for a nonzero vector $\theta$ supported on them.
Replacing $\theta$ by $-\theta$ if necessary, assume $\sum_j\theta_j\gr0$.
Then some $\theta_j>0$, and for $\tau=\min_{\theta_j>0}\lambda_j/\theta_j$ all coefficients $\lambda_j-\tau\theta_j$ are nonnegative, at least one vanishes, the represented vector is unchanged, and the total mass does not increase.
Iterating gives a representation using at most $r$ signed columns.
Padding the support if necessary, therefore
$$ M(B_1^{t+r})\subset \bigcup_{\substack{J\subset[t+r]\\|J|=r}} \conv\{\pm Me_j:j\in J\}. $$
Since $z_i\in B_2^r$ and $t\ls r$, every column of $M$ has Euclidean norm at most $\sqrt{r/t}$.
For $r$ vectors $w_1,\ldots,w_r$,
$$ \vol_r(\conv\{\pm w_1,\ldots,\pm w_r\})=\frac{2^r}{r!}|\det(w_1,\ldots,w_r)|. $$
The outer factor $2$ in the preceding inclusion multiplies volume by $2^r$.
Hence, using subadditivity of volume and Hadamard's inequality,
$$ \vol_r\left(Q'_S+t^{-1/2}B_2^r\right) \ls\binom{t+r}{r}\frac{4^r}{r!}\left(\frac{r}{t}\right)^{r/2} \ls\binom{t+r}{r}\left(\frac{4e}{\sqrt{rt}}\right)^r, $$
where the last inequality follows from $r!\gr(r/e)^r$.

For every measurable $A\subset\R^r$, the Gaussian density bound gives $\gamma_r(\rho TA)\ls(2\pi)^{-r/2}\rho^r(\det T)\vol_r(A)$.
Since $\det T=R_L(Q)^r$, the preceding volume estimate yields, for every $S$,
$$ \gamma_r\left(\rho T\left(Q'_S+t^{-1/2}B_2^r\right)\right) \ls\binom{t+r}{r} \left(\frac{C\rho R_L(Q)}{\sqrt{rt}}\right)^r. $$
Since the affine map is applied only after Gaussian measure has been bounded by density times volume, undoing the L\"owner normalization costs exactly $\det T$; no condition number enters.
Summing over the $\binom Nt$ sets in~\eqref{eq:Maureycover} proves~\eqref{eq:affine}.
This completes the proof.
\end{proof}

The next elementary suppression lemma is a variant of Friedland's random suppression lemma~\cite[Lemma~3.2]{Friedland26}.
The same suppression idea already appears in Tikhomirov's proof of Lemma~5.4~\cite{Tikhomirov19}, while Friedland isolates it as an averaging statement and combines it with block-tail Gram--Schmidt estimates in the small $U$ regime.
Here the lemma is used differently: no linear independence is required, the second family may have arbitrary cardinality and after suppression we quotient out the unsuppressed $K$ vectors and use Proposition~\ref{prop:wedge} followed by L\"owner normalization and Maurey's method.

\begin{lemma}\label{lem:suppression}
Let $ Q=\conv\bigl(\{\pm y_i:1\ls i\ls n\}\cup\{\pm z_j:1\ls j\ls N\}\bigr)\subset\R^n. $
Let $1\ls r\ls n$, put $\alpha=r/n$, and for $J\subset[n]$ with $|J|=r$ define $ Q_J=\conv\bigl(\{\pm y_i:i\notin J\}\cup\{\pm\alpha y_i:i\in J\}\cup\{\pm z_j:1\ls j\ls N\}\bigr). $
Then, for every $t>0$,
\begin{equation}\label{eq:suppression} \gamma_n(tQ)\ls\frac2{\binom nr}\sum_{|J|=r}\gamma_n(3tQ_J).\end{equation}
\end{lemma}

\begin{proof}
Fix $x\in tQ$ and choose a representation
$$ x=\sum_{i=1}^na_iy_i+\sum_{j=1}^Nb_jz_j,\qquad \sum_i|a_i|+\sum_j|b_j|\ls t. $$
For $J\subset[n]$, $|J|=r$, put $S_J=\sum_{i\in J}|a_i|$.
Averaging over all such $J$ gives
$$ \frac1{\binom nr}\sum_{|J|=r}S_J=\frac rn\sum_i|a_i|\ls\frac{rt}{n}. $$
Hence at least half of the sets $J$ satisfy $S_J\ls2rt/n$.
For each such $J$, the gauge of $x$ with respect to $Q_J$ is at most
$$ \sum_{i\notin J}|a_i|+\frac nr\sum_{i\in J}|a_i|+\sum_j|b_j|\ls t+\left(\frac nr-1\right)\frac{2rt}{n}\ls3t. $$
Thus $x\in3tQ_J$ for at least half of all $J$ and consequently
$$ \mathbf 1_{tQ}(x)\ls\frac2{\binom nr}\sum_{|J|=r}\mathbf 1_{3tQ_J}(x). $$
Integrating against $\gamma_n$ proves~\eqref{eq:suppression}.
\end{proof}

For $A\in\mathcal A_\eps$ put $ Q_A:=\conv\{\pm v_1,\ldots,\pm v_n,\pm u_1,\ldots,\pm u_n\}. $
Since $\Gamma A^{(i)}=v_i+u_i\in2Q_A$ for every $i$,
\begin{equation}\label{eq:KAQA} \mathcal K_A(\rho;\Gamma_A)\subset4\rho Q_A.\end{equation}
For $J\subset[n]$, $|J|=r$, put $\alpha=r/n$ and define $ Q_{A,J}:=\conv\bigl(\{\pm v_i:i\notin J\}\cup\{\pm\alpha v_i:i\in J\}\cup\{\pm u_j:1\ls j\ls n\}\bigr). $
Lemma~\ref{lem:suppression} and~\eqref{eq:KAQA} give
\begin{equation}\label{eq:KAsuppression} \gamma_n\bigl(\mathcal K_A(\rho;\Gamma_A)\bigr)\ls\frac2{\binom nr}\sum_{|J|=r}\gamma_n(12\rho Q_{A,J}).\end{equation}

\begin{proposition}\label{prop:half}
There are absolute constants $C_0,c_0,c>0$ with the following property.
Suppose
\begin{equation}\label{eq:mainconditions} \Lambda\gr2,\qquad r\gr2\Lambda,\qquad r^2\gr C_0sn\Lambda,\qquad \rho\delta\sqrt{\frac{\Lambda}{r}}\ls c_0,\end{equation}
where
\begin{equation}\label{eq:delta} \delta:=\max\left\{\frac rn,\frac1{\sqrt s}\right\}.\end{equation}
Then there is an event $\mathcal E$ with $\Pp(\mathcal E^c)\ls e^{-cr^2}$ such that, on $\mathcal E$, simultaneously for every $A\in\mathcal A_\eps$,
\begin{equation}\label{eq:half} \gamma_n\bigl(\mathcal K_A(\rho;\Gamma_A)\bigr)\ls\frac12.\end{equation}
\end{proposition}

\begin{proof}
Choose $C_0$ at least as large as the constant in Proposition~\ref{prop:wedge} and let $\mathcal E$ be the event of that proposition.
Its complement has probability at most $e^{-cr^2}$ by the third condition in~\eqref{eq:mainconditions}.
Fix $A\in\mathcal A_\eps$ and $J\subset[n]$ with $|J|=r$.
Let $B_J$ be the matrix whose columns are $a_i^{(K)}$, $i\notin J$.
Then $B_J\in\mathcal T_{n-r}$.
Let $\widehat E_{B_J}$ be the deterministic completion used in Proposition~\ref{prop:wedge} and put $H_J=(\Gamma\widehat E_{B_J})^\perp$.
For $i\notin J$, one has $a_i^{(K)}\in\operatorname{range}B_J\subset\widehat E_{B_J}$ and therefore $P_{H_J}v_i=0$.
Thus
$$ P_{H_J}Q_{A,J}=\conv\bigl(\{\pm\alpha P_{H_J}v_i:i\in J\}\cup\{\pm P_{H_J}u_j:1\ls j\ls n\}\bigr)=:\widetilde Q_{A,J}. $$
The space $H_J$ has dimension $r$ on $\mathcal E$.
Identify $H_J$ isometrically with $\R^r$.
If $\widetilde Q_{A,J}$ is not full dimensional in $H_J$ then $\gamma_{H_J}(12\rho\widetilde Q_{A,J})=0$.
Assume that it is full dimensional.

Take any $r$ generators $w_1,\ldots,w_r$ of $\widetilde Q_{A,J}$ and suppose that $p$ of them come from the suppressed $K$ vectors and $r-p$ from the $U$ vectors.
Recall from~\eqref{eq:delta} that $r/n\ls\delta$ and $s^{-1/2}\ls\delta$.
By multilinearity and Proposition~\ref{prop:wedge},
$$ \|w_1\wedge\cdots\wedge w_r\|\ls\left(\frac rn\right)^p(C\sqrt r)^rs^{-(r-p)/2}\ls(C\sqrt r\,\delta)^r. $$
Lemma~\ref{lem:loewner} therefore gives
\begin{equation}\label{eq:RL} R_L(\widetilde Q_{A,J})\ls C\sqrt r\,\delta.\end{equation}

Set $t=\lceil r/\Lambda\rceil$.
Since $r\gr2\Lambda$ and $\Lambda\gr2$ one has $  r/\Lambda\ls t\ls {2r}/{\Lambda}\ls r. $
The polytope $\widetilde Q_{A,J}$ has at most $n+r\ls2n$ generators.
Lemma~\ref{lem:affine} and~\eqref{eq:RL} give
$$ \gamma_{H_J}(12\rho\widetilde Q_{A,J})\ls\binom{n+r}{t}\binom{t+r}{r}\left(\frac{C\rho\delta}{\sqrt t}\right)^r. $$
Since $r\ls n$ and $t\ls2r/\Lambda$, using $\Lambda\gr\ln n$ we have
$$ \ln\binom{n+r}{t}\ls t\ln\frac{e(n+r)}t\ls\frac{2r}{\Lambda}\ln(2en)\ls Cr. $$
Also, since $t\ls r$, $ \binom{t+r}{r}\ls2^{t+r}\ls4^r. $ 
Using $t\gr r/\Lambda$,
\begin{equation}\label{eq:projectedmeasure} \gamma_{H_J}(12\rho\widetilde Q_{A,J})\ls\left(C\rho\delta\sqrt{\frac{\Lambda}{r}}\right)^r.\end{equation}
Choose $c_0$ sufficiently small in~\eqref{eq:mainconditions} so that the right hand side is at most $1/8$.

At this point $\Gamma$ is fixed and hence so is $H_J$.
If $G$ is an auxiliary standard Gaussian vector in $\R^n$ then $P_{H_J}G$ is standard Gaussian in $H_J$ and $G\in D$ implies $P_{H_J}G\in P_{H_J}D$.
Therefore
$ \gamma_n(12\rho Q_{A,J})\ls\gamma_{H_J}(12\rho\widetilde Q_{A,J})\ls 1/8. $
Returning to~\eqref{eq:KAsuppression} gives $ \gamma_n\bigl(\mathcal K_A(\rho;\Gamma_A)\bigr)\ls 1/4. $
In particular~\eqref{eq:half} holds.
Since Proposition~\ref{prop:wedge} holds simultaneously for every $B\in\mathcal T_{n-r}$ and all admissible vectors in $\mathcal D_s$ and $\mathcal C_s$, the same event $\mathcal E$ gives the conclusion for every $A\in\mathcal A_\eps$.
\end{proof}

\section{Proof of the main theorem}

\begin{proof}[Proof of Theorem~\ref{thm:main}]
Put $L=\ln n$ and choose
\begin{equation}\label{eq:parameters} \rho=c_1n^{5/8}L^{-5/8},\qquad s=\left\lfloor n^{1/2}L^{-1/2}\right\rfloor,\qquad r=\left\lceil C_2n^{3/4}L^{1/4}\right\rceil,\end{equation}
where $C_2$ will first be chosen sufficiently large and $c_1$ then sufficiently small.
Let $\Lambda$ be as in~\eqref{eq:lambda}.
For all sufficiently large $n$, $1\ls\rho\ls n$ and hence $ L\ls\Lambda\ls2L. $
In particular $\Lambda\gr2$.
Also $1\ls s\ls n$, $1\ls r\ls n$ and $r\gr2\Lambda$.

It remains to verify the remaining two conditions in~\eqref{eq:mainconditions}.
For large $n$, $ sn\Lambda\ls2n^{3/2}L^{1/2}$ and $ r^2\gr C_2^2n^{3/2}L^{1/2}. $
Thus, after choosing $C_2$ sufficiently large, $ r^2\gr C_0sn\Lambda. $
Moreover the quantity inside the floor defining $s$ tends to infinity so $ s\gr (1/2)n^{1/2}L^{-1/2}. $
Consequently $ r/n\ls CC_2n^{-1/4}L^{1/4}$ and $ 1/{\sqrt s}\ls Cn^{-1/4}L^{1/4}$ and hence $ \delta\ls CC_2n^{-1/4}L^{1/4}. $
Also $ \sqrt{{\Lambda}/{r}}\ls CC_2^{-1/2}n^{-3/8}L^{3/8}. $
Therefore
$$ \rho\delta\sqrt{\frac{\Lambda}{r}}\ls\left(c_1n^{5/8}L^{-5/8}\right)\left(CC_2n^{-1/4}L^{1/4}\right)\left(CC_2^{-1/2}n^{-3/8}L^{3/8}\right)\ls Cc_1C_2^{1/2}. $$
After $C_2$ has been fixed, choose $c_1>0$ sufficiently small.
Then all the conditions in~\eqref{eq:mainconditions} hold for all sufficiently large $n$.

Let $\mathcal E$ be the event of Proposition~\ref{prop:half}.
Then
\begin{equation}\label{eq:goodprob} \Pp(\mathcal E^c)\ls e^{-cr^2},\end{equation}
and on $\mathcal E$ simultaneously for every $A\in\mathcal A_\eps$ $ \gamma_n\bigl(\mathcal K_A(\rho;\Gamma_A)\bigr)\ls 1/2. $
For fixed $A\in\mathcal A_\eps$ define
$$ \mathcal H_A:=\left\{\gamma_n\bigl(\mathcal K_A(\rho;\Gamma_A)\bigr)\ls\frac12\right\}. $$
The event $\mathcal H_A$ depends only on the exposed Gaussian columns and is therefore $\mathcal F_A$ measurable; the global event $\mathcal E$ need not be.
This is the same measurability step used in Friedland's final powering argument~\cite[proof of Theorem~1.2]{Friedland26}.
Proposition~\ref{prop:half} gives $\mathcal E\subset\mathcal H_A$.
Hence~\eqref{eq:powering} implies
\begin{align*}
\Pp\bigl(&G_m\subset2\rho\Gamma A(B_1^n),\ \mathcal E\bigr) \ls \Pp\bigl(G_m\subset2\rho\Gamma A(B_1^n),\ \mathcal H_A\bigr)\\
&=\mathbb E\left[\mathbf1_{\mathcal H_A}\Pp\bigl(G_m\subset2\rho\Gamma A(B_1^n)\mid\mathcal F_A\bigr)\right] \ls\mathbb E\left[\mathbf1_{\mathcal H_A}\gamma_n\bigl(\mathcal K_A(\rho;\Gamma_A)\bigr)^{N(A)}\right] \ls2^{-N(A)}\ls2^{-(n^3-n^2)}.
\end{align*}
Using~\eqref{eq:Anet} and summing over $A\in\mathcal A_\eps$ gives $ \Pp\left(\exists A\in\mathcal A_\eps:G_m\subset2\rho\Gamma A(B_1^n),\ \mathcal E\right)\ls\exp(Cn^2\Lambda)2^{-(n^3-n^2)}\ls \exp{(-cn^3)}, $
where the last inequality follows from $\Lambda\ls2L$ and $n^2L=o(n^3)$.
By the discretization,
$$ \{\BM(G_m,B_1^n)\ls\rho\}\subset\mathcal E^c\cup\left(\mathcal E\cap\bigcup_{A\in\mathcal A_\eps}\{G_m\subset2\rho\Gamma A(B_1^n)\}\right). $$
Combining this inclusion with~\eqref{eq:goodprob} gives $ \Pp\{\BM(G_m,B_1^n)\ls\rho\}\ls \exp{(-cr^2)}+\exp{(-c'n^3)}\ls 2/n $
for all sufficiently large $n$.
Thus, with $c=c_1$, $ \Pp\left\{\BM(G_m,B_1^n)\gr c n^{5/8}(\ln n)^{-5/8}\right\}\gr1-2/n. $
Finally, $G_m$ is full dimensional almost surely so the event above has positive probability and therefore
$$ R_\infty(n)\gr c n^{5/8}(\ln n)^{-5/8}. $$
\end{proof}

The choice of parameters above is optimal, at polynomial precision, under the two conditions used in the proof.
Indeed write $s=n^a$, $r=n^b$ and $\rho=n^c$, and ignore logarithmic factors.
The entropy condition $r^2\gr C_0sn\Lambda$ from Proposition~\ref{prop:wedge} which ensures that the $e^{-cr^2}$ determinant tail absorbs the cardinalities of the coefficient classes gives $2b\gr1+a$ while the last condition in~\eqref{eq:mainconditions} gives
$$ c\ls\min\left\{1-\frac b2,\frac{a+b}{2}\right\}. $$
The largest possible value is attained at $a=1/2$ and $b=3/4$, and is $c=5/8$.
Thus a larger polynomial exponent requires an improvement of one of the structural estimates in the argument rather than a different choice of the parameters $r$ and $s$.

\bigskip

\noindent {\bf Acknowledgement.} 
The author acknowledges support by a PhD scholarship from the National Technical University of Athens.
The author is grateful to Apostolos Giannopoulos for useful discussions.
The author thanks Omer Friedland for sharing his independent manuscript before public posting and for a collegial exchange about the relation between the two arguments.

\bigskip 


\footnotesize
\bibliographystyle{amsplain}

\bigskip

\thanks{\noindent {\bf Keywords:} Banach--Mazur distance; Gaussian polytopes; Cross-polytope; L\"owner ellipsoid; Exterior products.}

\thanks{\noindent {\bf 2020 MSC:} Primary 52A23; Secondary 46B20, 60D05.}

\bigskip

\bigskip 

\medskip 

\noindent \textsc{Antonios \ Hmadi}: School of Applied Mathematical and Physical Sciences, National Technical University of Athens, Department of Mathematics, Zografou Campus, GR-157 80, Athens, Greece.

\smallskip

\noindent \textit{E-mail:} \texttt{ahmadi@mail.ntua.gr}

\end{document}